\documentclass[11pt]{article}
\usepackage[top=27mm, bottom=27mm, left=22mm, right=22mm]{geometry}
\usepackage[square, comma, sort&compress, numbers]{natbib}
\usepackage{hyperref}
\usepackage{amsfonts}
\usepackage{mathrsfs}
\usepackage{comment}
\usepackage{amsmath}
\usepackage{amssymb}
\usepackage{amsthm}
\usepackage{amscd}
\usepackage{graphicx}
\usepackage{indentfirst}
\usepackage[all]{xy}
\usepackage{titlesec}
\usepackage{enumerate}
\usepackage{bm}
\usepackage{enumitem}
\usepackage{color}
\usepackage{dsfont}
\usepackage{arydshln}
\usepackage[capitalize]{cleveref}
\usepackage{appendix}
\usepackage{ulem}
\newtheorem{theorem}{Theorem}[section]
\newtheorem{lemma}[theorem]{Lemma}

\newtheorem{conjecture}[theorem]{Conjecture}
\newtheorem{definition}[theorem]{Definition}
\newtheorem{construction}[theorem]{Construction}

\newtheorem{claim}[theorem]{Claim}

\newtheorem{question}[theorem]{Question}

\newtheorem{fact}[theorem]{Fact}

\newcommand{\ma}{\mathcal}

\newcommand{\mr}{\mathscr}

\newcommand{\ns}{\nsubseteq}

\newcommand{\fr}{\frac}
\newcommand{\lc}{\lceil}
\newcommand{\rc}{\rceil}
\newcommand{\lf}{\lfloor}
\newcommand{\rf}{\rfloor}

\newcommand{\ep}{\epsilon}

\newcommand{\e}{{\rm{ex}}}

\newcommand{\blue}{\textcolor{blue}}

\begin{document}

\title{Asymptotically sharp bounds for cancellative and union-free hypergraphs}

\date{}

\makeatletter
\def\thanks#1{\protected@xdef\@thanks{\@thanks
        \protect\footnotetext{#1}}}
\makeatother

\author{Miao Liu, Chong Shangguan, and Chenyang Zhang \thanks{The authors are with Research Center for Mathematics and Interdisciplinary Sciences, Shandong University, Qingdao 266237, China, and Frontiers Science Center for Nonlinear Expectations, Ministry of Education, Qingdao 266237, China (Emails: liumiao10300403@163.com, theoreming@163.com, lxzcyang@163.com)}
}

\maketitle

\begin{abstract}
    \noindent An $r$-graph is called $t$-cancellative if for arbitrary $t+2$ distinct edges $A_1,\ldots,A_t,B,C$, it holds that $(\cup_{i=1}^t A_i)\cup B\neq (\cup_{i=1}^t A_i)\cup C$; it is called $t$-union-free if for arbitrary two distinct subsets $\mathcal{A},\mathcal{B}$, each consisting of at most $t$ edges, it holds that $\cup_{A\in\mathcal{A}} A\neq \cup_{B\in\mathcal{B}} B$. Let $C_t(n,r)$ and $U_t(n,r)$ denote the maximum number of edges that can be contained in an $n$-vertex $t$-cancellative and $t$-union-free $r$-graph, respectively. The study of $C_t(n,r)$ and $U_t(n,r)$ has a long history, dating back to the classic works of Erd\H{o}s and Katona, and Erd\H{o}s and Moser in the 1970s. In 2020, Shangguan and Tamo showed that $C_{2(t-1)}(n,tk)=\Theta(n^k)$ and $U_{t+1}(n,tk)=\Theta(n^k)$ for all $t\ge 2$ and $k\ge 2$. In this paper, we determine the asymptotics of these two functions up to a lower order term, by showing that for all $t\ge 2$ and $k\ge 2$,
    \begin{align*}
        \text{$\lim_{n\rightarrow\infty}\frac{C_{2(t-1)}(n,tk)}{n^k}=\lim_{n\rightarrow\infty}\frac{U_{t+1}(n,tk)}{n^k}=\frac{1}{k!}\cdot \frac{1}{\binom{tk-1}{k-1}}$.}
    \end{align*}
    Previously, it was only known by a result of F\"uredi in 2012 that $\lim_{n\rightarrow\infty}\frac{C_{2}(n,4)}{n^2}=\frac{1}{6}$.

    To prove the lower bounds of the limits, we utilize a powerful framework developed recently by Delcourt and Postle, and independently by Glock, Joos, Kim, K\"uhn, and Lichev, which shows the existence of near-optimal hypergraph packings avoiding certain small configurations, and to prove the upper bounds, we apply a novel counting argument that connects $C_{2(t-1)}(n,tk)$ to a classic result of Kleitman and Frankl on a special case of the famous Erd\H{o}s Matching Conjecture.
\end{abstract}

\section{Introduction}

\noindent Given a configuration $\mr{L}$ defined by a family of $r$-uniform hypergraphs ($r$-graphs for short), an $r$-graph is said to be {\it $\mr{L}$-free} if it contains no copy of any member in $\mr{L}$. The {\it extremal number} $\e_r(n,\mr{L})$ is the maximum number of edges that can be contained in a $\mr{L}$-free $r$-graph on $n$ vertices. The study of extremal numbers of hypergraphs has been playing a central role in extremal combinatorics since the influential work of Tur\'an \cite{Turan-theorem}. Katona, Nemetz, and Simonovits \cite{Katona-Nemetz-Simonovits-Turan-density} showed that for every $\mr{L}$, the {\it Tur\'an density}  $\pi_r(\mr{L}):=\lim_{n\rightarrow\infty}\frac{\e_r(n,\mr{L})}{\binom{n}{r}}$ always exists. If $\pi_r(\mr{L})>0$ then we call $\mr{L}$ {\it non-degenerate}; otherwise, we call it {\it degenerate}. For a degenerate $\mr{L}$, if there exists a real number $\alpha\in(0,r)$ such that the limit $$\pi_r^d(\mr{L}):=\lim_{n\rightarrow\infty}\frac{\e_r(n,\mr{L})}{n^{\alpha}}$$ exists, then we call $\alpha$ and $\pi_r^d(\mr{L})$ the {\it Tur\'an exponent} and the {\it degenerate Tur\'an density} of $\mr{L}$-free $r$-graphs, respectively.

Unlike non-degenerate Tur\'an densities, even if we know the Tur\'an exponent of a degenerate $\mr{L}$, it is still a difficult task to determine whether $\pi_r^d(\mr{L})$ exists, not to mention computing its precise value. For example, when $r=2$, it is well-known that $\e_2(n,C_{2k})=\Theta(n^{1+1/k})$ for $k\in\{2,3,5\}$ (see \cite{wenger1991extremal, conlon2021extremal, bollobas2013modern}), and $\e_2(n,K_{s,t})=\Theta(n^{2-1/s})$ for $t\ge s=2,3$ and $t\ge (s-1)!+1$ (see \cite{kollar1996norm, alon1999norm}); however, it is only known that $\lim_{n\rightarrow\infty}\frac{\e_2(n,C_4)}{n^{3/2}}=\lim_{n\rightarrow\infty}\frac{\e_2(n,K_{2,t})}{n^{3/2}}$ exists and has precise value $\frac{1}{2}$. The reader is referred to \cite{Furedi-Simonovits-history} for a comprehensive survey for the case of $r=2$. For $r\ge 3$, a notable example in this direction is the celebrated work of R\"odl \cite{Rodl-nibble}, which asymptotically determines the maximum cardinality of an $(n,r,k)$-packing, where an {\it $(n,r,k)$-packing} is an $r$-graph on $n$ vertices such that every pair of distinct edges have less than $k$ common vertices. In the literature, there are several other works with similar flavor, including cover-free families \cite{Erdos-Frankl-Furedi-r-cover-free,Frankl-Furedi-colored-packing}, disjoint systems \cite{Alon-Sudakov-1995-disjoint-systems}, and $(v,e)$-free hypergraphs \cite{Glock64,Glock-Kim-Lichev-Pikhurko-Sun-bes,Letzter-Sgueglia-23,Shangguan-Tamo-degenerate}.



The main concern of this paper is to study the degenerate Tur\'an densities of cancellative and union-free hypergraphs, as defined below. For fixed positive integers $r,t$, an $r$-graph $\mathcal{H}$ is called

\begin{itemize}

    \item {\it $t$-cancellative}, if for arbitrary $t+2$ distinct edges $A_1,\ldots,A_t,B,C\in\mathcal{H}$, it holds that $(\cup_{i=1}^t A_i)\cup B\neq (\cup_{i=1}^t A_i)\cup C$;

    \item {\it $t$-union-free}, if for arbitrary two distinct subsets $\ma{A},\ma{B}\subseteq\mathcal{H}$, each consisting of at most $t$ edges, it holds that $\cup_{A\in\ma{A}} A\neq \cup_{B\in\ma{B}} B$.
\end{itemize}


\noindent Let $C_t(n,r)$ and $U_t(n,r)$ denote the maximum number of edges that can be contained in an $n$-vertex $t$-cancellative and $t$-union-free $r$-graph, respectively. For a large set of parameters, we are able to determine the asymptotics of $C_t(n,r)$ and $U_t(n,r)$ up to a lower order term.


\subsection{Cancellative $r$-graphs}

\noindent Erd\H{o}s and Katona \cite{Erdos-Katona-1-canc} initiated the study of $1$-cancellative hypergraphs in the 1970s. This notion has an interesting connection with coding theory (see \cite{tolhuizen-1-canc-low}) and is closely related to several branches of extremal combinatorics. It has motivated many follow-up works in related fields; see \cite{tolhuizen-1-canc-low,shearer1996new,Korner-Sinaimeri-2-canc,pikhurko2008exact,Frankl-Furedi-1-canc-upp,Shangguan-Tamo-canc-and-union-free,Furedi-2-canc,Keevash-Mubayi-canc,Liu-canc,Ni-Liu-Kang-canc} for some examples. Since every $r$-partite $r$-graph is $1$-cancellative, we have
\begin{align*}
    C_1(n,r)\ge\left\lfloor\frac{n}{r}\right\rfloor\times\left\lfloor\frac{n+1}{r}\right\rfloor\times\cdots\left\lfloor\frac{n+r-1}{r}\right\rfloor:=p(n,r),
\end{align*}

\noindent where $p(n,r)$ is the number of edges in a balanced complete $r$-partite $r$-graph on $n$ vertices. For $r=2$, Mantel's theorem \cite{Mantel-K_3} implies that $C_1(n,2)=\e_2(n,K_3)=p(n,2)$. For $r=3,4$, Bollob\'as \cite{bollobas1974three} and Sidorenko \cite{sidorenko1987maximal} showed that $C_1(n,3)=p(n,3)$ and $C_1(n,4)=p(n,4)$, respectively. Frankl and F\"uredi \cite{Frankl-Furedi-1-canc-upp} showed that $C_1(n,r)=p(n,r)$ for all $2r\ge n\ge r$. However, Shearer \cite{shearer1996new} showed that the equality was no longer true for $r\ge 11$ and sufficiently large $n$. Currently, it is known that $$\fr{0.28}{2^r}\binom{n}{r}<C_1(n,r)\le\fr{2^r}{\binom{2r}{r}}\binom{n}{r},$$
where the lower and upper bounds were due to Tolhuizen \cite{tolhuizen-1-canc-low} and Frankl and F\"uredi \cite{Frankl-Furedi-1-canc-upp}, respectively.

K\"orner and Sinaimeri \cite{Korner-Sinaimeri-2-canc} first studied 2-cancellative hypergraphs, and F\"uredi \cite{Furedi-2-canc} introduced the general definition for $t$-cancellative hypergraphs. For even $r$, F\"uredi \cite{Furedi-2-canc} showed that $C_2(n,4)=\frac{n^2}{6}+o(n^2)$, and more generally,
\begin{align}\label{eq:2-canc}
    \frac{n^k}{(2k)^k}-o(n^k)\le C_2(n,2k)\le\frac{\binom{n}{k}}{\binom{2k-1}{k-1}}.
\end{align}
\noindent For odd $r$, it is known that $n^{k-o(1)}<C_2(n,2k-1)=O(n^{k})$ (see \cite{Shangguan-Tamo-canc-and-union-free} for the lower bound and \cite{Furedi-2-canc} for the upper bound).

In a recent work, Shangguan and Tamo \cite{Shangguan-Tamo-canc-and-union-free} showed that for all $t\ge 2$ and $r\ge 3$,
\begin{align}\label{eq:ST-canc-bound}
    \Omega(n^{\lf\frac{2r}{t+2}\rf+\frac{2r(\bmod t+2)}{t+1}})=C_t(n,r)=O(n^{\lceil\frac{r}{\lfloor t/2\rfloor+1}\rceil}),
\end{align}

\noindent which implies that $C_{2(t-1)}(n,tk)=\Theta(n^k)$, thereby determining the  Tur\'an exponent of $C_{2(t-1)}(n,tk)$. It is therefore an interesting problem to study whether the degenerate Tur\'an density $\lim_{n\rightarrow\infty}\frac{C_{2(t-1)}(n,tk)}{n^k}$ exists. Indeed, for $t=2$, F\"uredi \cite{Furedi-2-canc} suspected that the upper bound in \eqref{eq:2-canc} might be close to the truth, and he proved it for the special case $C_2(n,4)$.

We confirm F\"uredi's prediction for all $C_2(n,2k)$ with $k\ge 3$. More generally, we completely determine the degenerate Tur\'an density of $C_{2(t-1)}(n,tk)$.

\begin{theorem}\label{thm:main-canc}
    For all $t\ge 2$ and $k\ge 2$,
    \begin{align}\label{eq:canc}
        \lim_{n\rightarrow\infty}\frac{C_{2(t-1)}(n,tk)}{n^k}=\frac{1}{k!}\cdot \frac{1}{\binom{tk-1}{k-1}}.
    \end{align}
\end{theorem}

\subsection{Union-free $r$-graphs}

\noindent Motivated by applications in information retrieval and data communication, Kautz and Singleton \cite{kautz-singleton-superimposed-codes} introduced the notion of $t$-union-free hypergraphs in 1964\footnote{with the name ``uniquely decipherable codes of order $t$''}. Unaware of their work, Erd\H{o}s and Moser \cite{Erdos-Moser-2-union-free} reintroduced the 2-union-free hypergraphs in 1970.


Union-free hypergraphs are closely related to cover-free hypergraphs, which were introduced independently by Kautz and Singleton \cite{kautz-singleton-superimposed-codes} and Erd\H{o}s, Frankl, and F\"uredi \cite{Erdos-Frankl-Furedi-2-cover-free,Erdos-Frankl-Furedi-r-cover-free}. An $r$-graph $\ma{H}$ is said to be {\it $t$-cover-free} if for arbitrary $t+1$ distinct edges $A_1,\ldots,A_t,B\in\ma{H}$, it holds that $B\nsubseteq \cup_{i=1}^t A_i.$


The following fact is well-known (see, e.g. \cite{Furedi-Ruszinko-no-grid}).
\paragraph{Fact.} If $\ma{H}$ is $t$-cover-free then it is $t$-union-free; if $\ma{H}$ is $t$-union-free then it is $(t-1)$-cover-free.

Let $F_t(n,r)$ denote the maximum number of edges that can be contained in an $n$-vertex $t$-cover-free $r$-graph. By the above fact, we have
\begin{align}\label{eq:union-free-and-cover-free}
    F_t(n,r)\le U_t(n,r)\le F_{t-1}(n,r).
\end{align}

\noindent Frankl and F\"uredi \cite{Frankl-Furedi-colored-packing} showed that there is a constant $\gamma(r,t)$ depending only on $r,t$ such that
\begin{equation}\label{eq:cff-frankl-furedi}
    \lim_{n\to \infty}\frac{F_{t}(n,r)}{n^{\lc\fr{r}{t}\rc}}=\gamma(r,t).
\end{equation}

\noindent Note that the $\gamma(r,t)$ in \eqref{eq:cff-frankl-furedi} is defined by the extreme of the well-known Erd\H{o}s Matching conjecture \cite{erdos1965problem}. In particular, it is known that (see e.g. the discussion in \cite{Frankl-Furedi-colored-packing}) 
\begin{align}\label{eq:cff-frankl-furedi-2}
    \lim_{n\to \infty}\frac{F_{t}(n,tk)}{n^k}=\frac{1}{k!}\cdot \frac{1}{\binom{tk-1}{k-1}}.
\end{align}

\noindent Observe that \eqref{eq:union-free-and-cover-free} and \eqref{eq:cff-frankl-furedi} together yield that
\begin{equation}\label{eq:cff-bound}
    \Omega(n^{\lc\fr{r}{t}\rc})=F_t(n,r)\le U_t(n,r)\le F_{t-1}(n,r)=O(n^{\lc\fr{r}{t-1}\rc}).
\end{equation}

There have been several improvements on \eqref{eq:cff-bound} in the past forty years. Frankl and F\"uredi \cite{Frankl-Furedi-2-union-free} showed that $U_2(n,r)=\Theta(n^{\lc 4r/3\rc/2})$. F\"uredi and Ruszink\'{o} \cite{Furedi-Ruszinko-no-grid} showed that $n^{2-o(1)}<U_r(n,r)=O(n^2)$. Recently, Shangguan and Tamo \cite{Shangguan-Tamo-canc-and-union-free} showed that for all $t\ge 3$ and $r\ge 3$,
\begin{equation}\label{eq:Shangguan-Tamo-union-free}
    U_t(n,r)=\Omega(n^{\fr{r}{t-1}}).
\end{equation}

\noindent Then, \eqref{eq:union-free-and-cover-free}, \eqref{eq:cff-frankl-furedi-2}, and \eqref{eq:Shangguan-Tamo-union-free} yield that
\begin{align}\label{eq:ST-union-free}
    \Omega(n^k)=U_{t+1}(n,tk)\le F_t(n,tk)\le\frac{1}{k!}\cdot \frac{1}{\binom{tk-1}{k-1}}\cdot n^k+o(n^k).
\end{align}

We present a new lower bound for $U_{t+1}(n,tk)$, showing that the upper bound in \eqref{eq:ST-union-free} is almost tight.

\begin{theorem}\label{thm:main-union-free}
    For all $t\ge 2$ and $k\ge 2$,
    \begin{align}\label{eq:union-free}
        \lim_{n\to \infty}\frac{U_{t+1}(n,tk)}{n^k}=\frac{1}{k!}\cdot \frac{1}{\binom{tk-1}{k-1}}.
    \end{align}
\end{theorem}


\paragraph{Overview of the proof of the lower bounds.} Note that the limits in Theorems \ref{thm:main-canc} and \ref{thm:main-union-free} magically coincide. Indeed, we will prove the two lower bounds $C_{2(t-1)}(n,tk)\ge(1-o(1))\cdot \binom{n}{k} / \binom{tk-1}{k-1}$ and $U_{t+1}(n,tk)\ge (1-o(1)) \cdot\binom{n}{k} / \binom{tk-1}{k-1}$, where $o(1)\rightarrow 0$ as $n\rightarrow\infty$, in a quite unified way. We will construct the desired cancellative (resp. union-free) hypergraph based on a framework that is widely used in proving the existence or computing the precise values of degenerate Tur\'an densities, see \cite{Erdos-Frankl-Furedi-r-cover-free,Frankl-Furedi-colored-packing,Glock64,Glock-Kim-Lichev-Pikhurko-Sun-bes,Letzter-Sgueglia-23,Shangguan-Tamo-degenerate} for a few examples.

Let $\mathcal{F}$ be a fixed $tk$-graph. The {\it $k$-shadow} of $\mathcal{F}$, denoted as $\mathcal{J}(\mathcal{F})$, consists of all $k$-subsets of $V(\mathcal{F})$ that are contained in at least one member of $\mathcal{F}$; in other words, $\mathcal{J}(\mathcal{F})=\cup_{F\in\mathcal{F}}\binom{F}{k}$.

\begin{construction}\label{construction}
Let $\mathcal{F}$ be a fixed $tk$-graph and $\mathcal{J}(\mathcal{F})$ be its $k$-shadow. Let $\mathcal{P}=\{(V(\mathcal{J}_i),\mathcal{J}_i):i\in[|\mathcal{P}|]\}$ be a {\it $\mathcal{J}(\mathcal{F})$-packing} in $\binom{[n]}{k}$, where the $\mathcal{J}_i$'s are pairwise edge-disjoint copies of $\mathcal{J}(\mathcal{F})$ with vertex set $V(\mathcal{J}_i)$. Put a copy $\mathcal{F}_i$ of $\mathcal{F}$ on top of each $\mathcal{J}_i$, and let
    \begin{align}\label{eq:H-sparse1}
        \mathcal{H}(\mathcal{F})=\bigcup_{i=1}^{|\mathcal{P}|} \mathcal{F}_i.
    \end{align}
\noindent Then for each $i\in[|\mathcal{P}|]$, we have $\mathcal{J}(\mathcal{F}_i)=\mathcal{J}_i$ and $V(\mathcal{F}_i)=V(\mathcal{J}_i)$; moreover, $\mathcal{H}(\mathcal{F})\subseteq\binom{[n]}{tk}$ is a $tk$-graph formed by the union of copies of $\mathcal{F}$.
\end{construction}

Note that we can assume that the $\mathcal{J}(\mathcal{F})$-packing $\mathcal{P}$ is near-optimal in the sense that $|\mathcal{P}|\ge(1-o(1))\cdot\binom{n}{k}/|\mathcal{J}(\mathcal{F})|$, where $o(1)\rightarrow 0$ as $n\rightarrow\infty$. To prove the desired lower bounds, we will show that if $\mathcal{F}$ and $\mathcal{P}$ satisfy some nice properties, then the $tk$-graph $\mathcal{H}(\mathcal{F})$ in \cref{construction} must be $2(t-1)$-cancellative (resp. $(t+1)$-union-free). Crucially, we also require that $|\mathcal{F}|/|\mathcal{J}(\mathcal{F})|\ge (1-\delta)\cdot 1/\binom{tk-1}{k-1}$, where $\delta$ can be made arbitrarily small as $|V(\mathcal{F})|$ is sufficiently large. So we can conclude by \eqref{eq:H-sparse1} that
\begin{align*}
    |\mathcal{H}(\mathcal{F})|=|\mathcal{P}|\cdot|\mathcal{F}|\ge(1-o(1))\cdot\binom{n}{k}|\mathcal{F}|/|\mathcal{J}(\mathcal{F})|\ge(1-o(1))\cdot \binom{n}{k} \bigg/ \binom{tk-1}{k-1}.
\end{align*}

\noindent The main difficulty is to set up the properties of $\mathcal{F}$ and $\mathcal{P}$ accordingly so that $\mathcal{H}(\mathcal{F})$ is $2(t-1)$-cancellative (resp. $(t+1)$-union-free) and $|\mathcal{F}|/|\mathcal{J}(\mathcal{F})|$ satisfies the prescribed inequality. Fortunately, we find a way to overcome this difficulty. The details can be found in Lemmas \ref{lem:canc-LB} and \ref{lem:uf-LB} below.

\paragraph{Organization.} The rest of this paper is organised as follows.
We will present some definitions and some auxiliary lemmas in \cref{sec:definitions-lemmas}, and prove these lemmas in \cref{sec:proof-lemmas}. We will prove our main results, Theorems \ref{thm:main-canc} and \ref{thm:main-union-free}, in Sections \ref{sec:thm:main-canc} and \ref{sec:thm:main-union-free}, respectively. Lastly, we will conclude this paper and mention some open questions in \cref{sec:conclusion}.

\section{Definitions and lemmas}\label{sec:definitions-lemmas}

\noindent Throughout, we assume that $r,t,k,v,e$ are fixed positive integers and $n\rightarrow\infty$. We use $[n]$ to denote the set $\{1,\ldots,n\}$. For a finite set $V$, we use $\binom{V}{r}$ to denote the family of all $r$-subsets of $V$. An $r$-graph $\mathcal{H}$ on the vertex set $V(\mathcal{H})$ is a family of distinct $r$-subsets of $V(\mathcal{H})$, that is, $\mathcal{H}\subseteq\binom{V(\mathcal{H})}{r}$. Assume without loss of generality that every vertex in $V(\mathcal{H})$ is contained in at least one edge in $\mathcal{H}$.

We denote the {\it symmetric difference} of two finite subsets $B,C$ by $B\Delta C:=(B\setminus C)\cup(C\setminus B)$. With this notation, a hypergraph $\mathcal{H}$ is $t$-cancellative, if for arbitrary $t+2$ distinct edges $A_1,\ldots,A_t,B,C\in\mathcal{H}$, it holds that $B\Delta C\ns \cup_{i=1}^t A_i$.


An {\it $(v,e)$-configuration} is a family of $e$ distinct edges whose union contains at most $v$ vertices. We say an $r$-graph $\mathcal{H}$ is {\it $(v,e)$-free} if it contains no $(v,e)$-configuration, i.e., for arbitrary $e$ distinct edges $A_1,\ldots,A_e\in\mathcal{H}$, it holds that $|\cup_{i=1}^e A_i|\ge v+1$. For convenience, if there is no confusion on $r$ and $k$, then we use an $e$-configuration (resp. an $e^-$-configuration) to denote an $(er-(e-1)k,e)$-configuration (resp. an $(er-(e-1)k-1,e)$-configuration). With this notation, we say that an $r$-graph is {\it $e$-free} (resp. {\it $e^-$-free}) if it is $(er-(e-1)k,e)$-free (resp. $(er-(e-1)k-1,e)$-free).

The following lemma, which provides a lower bound on the maximum number of edges that can be contained in an $n$-vertex $(v,e)$-free $r$-graph, can be easily proved by the probabilistic method.

\begin{lemma}[see, e.g. Proposition 6 from \cite{Shangguan-Tamo-spa-hyp-coding}]\label{lem:e^--free}
    For all fixed integers $r>k\ge 2$ and $e\ge 2$, there exists an $r$-graph $\mathcal{G}\subseteq\binom{[m]}{r}$ with at least $cm^{k+\frac{1}{e-1}}$ edges that is $\ell^-$-free for all $2\le\ell\le e$, where $c=c(r,k,e)$ is an absolute constant independent of $m$.
\end{lemma}

Next, we present the formal definitions of packings and induced packings.

\begin{definition}[Packings and induced packings]\label{def:packing}
    For a fixed $k$-graph $\mathcal{J}$, a {\it $\mathcal{J}$-packing} $\mathcal{P}$ in $\binom{[n]}{k}$ is a family of {\it edge-disjoint} copies of $\mathcal{J}$, that is, $\mathcal{P}=\{(V(\mathcal{J}_i),\mathcal{J}_i):i\in [|\mathcal{P}|]\}$, where for each $i$, $\mathcal{J}_i\subseteq\binom{[n]}{k}$ is a copy of $\mathcal{J}$, and for distinct $i,i'$, $\mathcal{J}_i\cap\mathcal{J}_{i'}=\emptyset$.

    The $\mathcal{J}$-packing $\mathcal{P}$ is further said to be {\it induced} if for all distinct $i,i'$, $|V(\mathcal{J}_i)\cap V(\mathcal{J}_{i'})|\le k$, and if $|V(\mathcal{J}_i)\cap V(\mathcal{J}_{i'})|=k$ then $V(\mathcal{J}_i)\cap V(\mathcal{J}_{i'})$ is neither an edge in $\mathcal{J}_i$ nor an edge in $\mathcal{J}_{i'}$.
\end{definition}

Since the $k$-graphs in a $\mathcal{J}$-packing are pairwise edge-disjoint, every $\mathcal{J}$-packing in $\binom{[n]}{k}$ can have at most $\frac{\binom{n}{k}}{|\mathcal{J}|}$ copies of $\mathcal{J}$. The influential works of R\"odl \cite{Rodl-nibble}, Frankl and R\"odl \cite{Frankl-Rodl-matching}, and Pippenger (see \cite{Pippenger-Spencer-asymptotic}) showed that the upper bound is asymptotically tight in the sense that there exists a near-optimal $\mathcal{J}$-packing that contains at least $(1-o(1))\cdot\frac{\binom{n}{k}}{|\mathcal{J}|}$ edge-disjoint copies of $\mathcal{J}$. Frankl and F\"uredi \cite{Frankl-Furedi-colored-packing} strengthened their results by showing that there exists a near-optimal induced $\mathcal{J}$-packing. Delcourt and Postle \cite{Delcourtconflict} and independently Glock, Joos, Kim, K\"uhn, and Lichev \cite{Glockconflict} greatly extended the above line of work by establishing a general framework that can be used to show the existence of near-optimal $\mathcal{J}$-packings avoiding certain small configurations.

In order to construct cancellative and union-free hypergraphs, we will need $\mathcal{J}$-packings avoiding certain $(v,e)$-configurations. We apply the tools developed in \cite{Delcourtconflict,Glockconflict} to prove the following lemma, which shows that for every fixed $k$-graph $\mathcal{J}$, there exists a near-optimal induced  $\mathcal{J}$-packing in $\binom{[n]}{k}$ that does not contain $\ell^-$-configurations for every fixed small $\ell$.

\begin{lemma}\label{lem:conflict-free-matching}
    Let $m\ge k,~e\ge 2$ be fixed integers and $\mathcal{J}$ be a $k$-graph on $m$ vertices. Then there exists a $\mathcal{J}$-packing $\mathcal{P}=\{(V(\mathcal{J}_i),\mathcal{J}_i):i\in[|\mathcal{P}|]\}$ in $\binom{[n]}{k}$ such that the following conditions hold:
    \begin{itemize}
        \item [(i)] $\mathcal{P}$ is induced;
        \item [(ii)] $\mathcal{V}=\{V(\mathcal{J}_i):i\in[|\mathcal{P}|]\}$ is $\ell^-$-free (i.e., $(\ell\cdot m-(\ell-1)k-1,\ell)$-free) for all $2\le\ell\le e$;
        \item [(iii)] $|\mathcal{P}|\ge(1-o(1))\binom{n}{k}/|\mathcal{J}|$, where $o(1)\rightarrow 0$ as $n\rightarrow\infty$.
    \end{itemize}
\end{lemma}

We will also use the following lemma, which states a useful property of the hypergraph $\mathcal{H}(\mathcal{F})$ defined in \cref{construction}.

\begin{lemma}\label{lem:H(F)}
    Let $\mathcal{F}$ be a $tk$-graph on $m$ vertices and $\mathcal{J}(\mathcal{F})$ be the $k$-shadow of $\mathcal{F}$. Let $\mathcal{P}=\{(V(\mathcal{J}_i),\mathcal{J}_i):i\in[|\mathcal{P}|]\}$ be a {\it $\mathcal{J}(\mathcal{F})$-packing} in $\binom{[n]}{k}$. Suppose that $\mathcal{F}$ and $\mathcal{P}$ satisfy the following properties:
    \begin{itemize}
        \item [{\rm (i)}] $\mathcal{F}$ is $\ell^-$-free (i.e., $(\ell\cdot tk-(\ell-1)k-1,\ell)$-free) for all $2\le\ell\le e$;
        \item [{\rm (ii)}] $\mathcal{V}=\{V(\mathcal{J}_i):i\in[|\mathcal{P}|]\}$ is $\ell^-$-free (i.e., $(\ell\cdot m-(\ell-1)k-1,\ell)$-free) for all $2\le\ell\le e$.
    \end{itemize}
    \noindent Then the $tk$-graph $\mathcal{H}(\mathcal{F})$ defined in \cref{construction} is $\ell^-$-free (i.e., $(\ell\cdot tk-(\ell-1)k-1,\ell)$-free) for all $2\le\ell\le e$.
\end{lemma}

The proofs of Lemmas \ref{lem:conflict-free-matching} and \ref{lem:H(F)} are postponed to \cref{sec:proof-lemmas} below.

\section{Proofs of the lemmas}\label{sec:proof-lemmas}

\subsection{Proof of \cref{lem:conflict-free-matching}}

\noindent We will prove \cref{lem:conflict-free-matching} by using a powerful result proved independently in \cite{Glockconflict} and \cite{Delcourtconflict}. For real numbers $a$ and $b\in(0,1)$, we write $a=1\pm b$ if $a\in[1-b,1+b]$. For a hypergraph $\mathcal{C}$ (not necessarily uniform) and an integer $\ell\ge1$, we use $\mathcal{C}^{(\ell)}$ to denote the subhypergraph of $\mathcal{C}$ formed by the $\ell$-edges in $\mathcal{C}$, i.e., $V(\mathcal{C}^{(\ell)})=V(\mathcal{C})$ and $\mathcal{C}^{(\ell)}=\{C\in\mathcal{C}:|C|=\ell\}$. For a subset $X\subseteq V(\mathcal{C})$, we use $\deg_{\mathcal{C}}(X)$ to denote the number of edges in $\mathcal{C}$ that contain $X$, that is, $\deg_{\mathcal{C}}(X)=|\{C\in \mathcal{C}:X\subseteq C\}|$. Moreover, let $\Delta_{\ell}(\mathcal{C})$ be the maximum of $\deg_{\mathcal{C}}(X)$ over all $X\in\binom{V(\mathcal{C})}{\ell}$. Lastly, a {\it matching} in $\mathcal{C}$ is a set of pairwise disjoint edges in $\mathcal{C}$.

\begin{lemma}[see Theorem 1.3 in \cite{Glockconflict}, see also Corollary 1.17 in \cite{Delcourtconflict}]\label{lem:usefule-lem}
    For all integers $r,e\ge2$, there exists $\ep_0>0$ such that for all $\ep\in(0,\ep_0)$, there exists $d_0$ such that the following holds for all $d\ge d_0$. Let $\mathcal{H}$ be an $r$-graph with $|V(\mathcal{H})|\le\exp(d^{\ep^3})$ such that every vertex is contained in $(1\pm d^{-\ep})d$ edges and $\Delta_2(\mathcal{H})\le d^{1-\ep}$.

    Let $\mathcal{C}$ be a hypergraph with $V(\mathcal{C})=\mathcal{H}$ such that every $C\in\mathcal{C}$ satisfies $2\le|C|\le e$, and the following conditions hold:
    \begin{itemize}
        \item [(C1)] $\Delta_1(\mathcal{C}^{(\ell)})\le ed^{\ell-1}$ for all $2\le \ell\le e$;
        \item [(C2)] $\Delta_{\ell'}(\mathcal{C}^{(\ell)})\le d^{\ell-\ell'-\epsilon}$ for all $2\le \ell'<\ell\le e$;
        \item [(C3)] $|\{C_1\in\mathcal{H}:\{C_1,C_2\}\in\mathcal{C} \mbox{ and } v\in C_1\}|\le d^{1-\ep}$ for all $C_2\in \mathcal{H}$ and $v\in V(\mathcal{H})$;
        \item [(C4)] $|\{C_1\in\mathcal{H}:\{C_1,C_2\},\{C_1,C_3\}\in\mathcal{C}\}|\le d^{1-\ep}$ for all disjoint $C_2,C_3\in \mathcal{H}$.
    \end{itemize}
    Then, there exists a $\mathcal{C}$-free matching $\mathcal{M}\subseteq\mathcal{H}$ which covers all but $d^{-\ep^3}|V(\mathcal{H})|$ vertices of $\mathcal{H}$.
\end{lemma}

Now, we will derive \cref{lem:conflict-free-matching} from \cref{lem:usefule-lem}.

\begin{proof}[Proof of \cref{lem:conflict-free-matching}]
     Consider the complete $k$-graph $\binom{[n]}{k}$. Color every edge of $\binom{[n]}{k}$ red with probability $\epsilon$ and blue with probability $1-\epsilon$, where $\epsilon\in(0,\frac{1}{2(m-k)})$ is a sufficiently small but fixed constant. All colorings are independent. Let $\mathcal{H}$ be a hypergraph whose vertices are the blue edges in $\binom{[n]}{k}$. A copy $\mathcal{J}'$ of $\mathcal{J}$ with $V(\mathcal{J}')\subseteq[n]$ forms an edge in $\mathcal{H}$ if and only if every $k$-subset in $\mathcal{J}'$ is colored blue and every $k$-subset in $\binom{V(\mathcal{J}')}{k}/\mathcal{J}'$ is colored red. Note that $\mathcal{H}$ is a $|\mathcal{J}|$-graph. Later, one can see that such a coloring process yields an ``induced'' $\mathcal{J}$-packing that proves \cref{lem:conflict-free-matching} (i).

     To establish \cref{lem:conflict-free-matching} (ii), we define $\mathcal{C}$ as the collection of copies of $\mathcal{J}$ in $\mathcal{H}$ whose vertex sets form an {\it minimal} bad configuration, that is, their vertex sets form an $\ell^-$-configuration for some $2\le \ell\le e$ but contain no $\ell'^-$-configuration for any $1\le \ell'< \ell$. More precisely,
     \begin{equation}\label{eq:def-C}
         \begin{aligned}
         \mathcal{C}:=\bigcup_{\ell=2}^{e}\Bigg\{\{\mathcal{J}_1,\ldots,\mathcal{J}_{\ell}\} \subseteq\mathcal{H}:&\left|\bigcup_{i=1}^{\ell}V(\mathcal{J}_i)\right|\le\ell m-(\ell-1)k-1,\\
         &\left|\bigcup_{i\in S}V(\mathcal{J}_i)\right|\ge|S|m-(|S|-1)k \mbox{ for all }S\subsetneq [\ell]\Bigg\}.
     \end{aligned}
     \end{equation}

    The following claim connects matchings in $\mathcal{H}$ to $\mathcal{J}$-packings in $\binom{[n]}{k}$.

    \begin{claim}\label{claim:matching-packing}
        Every $\mathcal{C}$-free matching in $\mathcal{H}$ yields a $\mathcal{J}$-packing in $\binom{[n]}{k}$ that satisfies \cref{lem:conflict-free-matching} (i), (ii).
    \end{claim}

    \begin{proof}
        The claim above can be verified straightforwardly by definition. To see this, let $\mathcal{M}\subseteq\mathcal{H}$ be a $\mathcal{C}$-free matching. Note that the edges in $\mathcal{M}$ are also copies of $\mathcal{J}$. These copies are pairwise edge disjoint as $k$-graphs in $\binom{[n]}{k}$, since they are pairwise vertex disjoint as edges (of size $|\mathcal{J}|$) in $\mathcal{H}$. Therefore, the edges in $\mathcal{M}$ yield a $\mathcal{J}$-packing in $\binom{[n]}{k}$, denoted by $\mathcal{P}$. As $\mathcal{M}$ is $\mathcal{C}$-free, $\mathcal{P}$ automatically satisfies \cref{lem:conflict-free-matching} (ii).

    It remains to show that $\mathcal{P}$ also satisfies \cref{lem:conflict-free-matching} (i), i.e., it is induced. Indeed, let $\mathcal{J}'$ and $\mathcal{J}''$ be two distinct copies of $\mathcal{J}$ in $\mathcal{P}$. As $\mathcal{P}$ satisfies \cref{lem:conflict-free-matching} (ii) with $\ell=2$, it is clear that $|V(\mathcal{J}')\cap V(\mathcal{J}'')|\le k$. Suppose that $|V(\mathcal{J}')\cap V(\mathcal{J}'')|=k$. We will show that $V(\mathcal{J}')\cap V(\mathcal{J}'')$ is neither an edge in $\mathcal{J}'$ nor an edge in $\mathcal{J}''$. Otherwise, assume that $V(\mathcal{J}')\cap V(\mathcal{J}'')$ is an edge in $\mathcal{J}'$. Then by definition, it is colored blue in $\mathcal{J}'$. However, as $\mathcal{J}'$ and $\mathcal{J}''$ are edge disjoint copies of $\mathcal{J}$, $V(\mathcal{J}')\cap V(\mathcal{J}'')$ has to be colored red in $\mathcal{J}''$, which is impossible.
    \end{proof}

    Given \cref{claim:matching-packing}, to prove \cref{lem:conflict-free-matching}, it suffices to show that there exists a near-optimal $\mathcal{C}$-free matching in $\mathcal{H}$. We will prove it by applying \cref{lem:usefule-lem} with $\epsilon\in(0,\frac{1}{2(m-k)})$ and $d=d_0=cn^{m-k}$ where $c$ is constant depending on $\mathcal{J}$, $\epsilon$, $m$ and $k$ (see the proof of \cref{claim-concentration} for details). For this purpose, below we will show that $\mathcal{H}$ and $\mathcal{C}$ satisfy the assumptions of \cref{lem:usefule-lem}.

    First, let us consider $\mathcal{H}$. By the definition of $\mathcal{H}$, $|V(\mathcal{H})|$ can be viewed as the sum of $n\choose k$ independent indicator random variables, i.e.,
    $$|V(\mathcal{H})|=\sum_{K\in \binom{[n]}{k}}X_{K},$$
    where $X_{K}$ is the indicator random variable for the event that $K$ is colored blue. Therefore, $\mathbb{E}[|V(\mathcal{H})|]=\sum_{K\in \binom{[n]}{k}}\Pr[X_{K}=1]=(1-\epsilon)\binom{n}{k}$. By Hoeffding inequality (see, e.g., \cite{alon2016probabilistic}), one can show that with high probability $|V(\mathcal{H})|=(1\pm 2\epsilon)\binom{n}{k}$. Moreover, for sufficiently large $n$ we have $|V(\mathcal{H})|\le n^k\le\mbox{exp}(d^{\ep^3})$. The following claim shows that $\mathcal{H}$ is approximately regular.
    \begin{claim}\label{claim-concentration}
         With high probability, every vertex in $\mathcal{H}$ is contained in $(1\pm d^{-\epsilon})d$ edges.
    \end{claim}

   \cref{claim-concentration} can be proved by the well-known bounded difference inequality (see \cref{lem:BDI}) and we postpone its proof to \cref{appendix-concentration}.

   For arbitrary two distinct vertices $T_1,T_2\in V(\mathcal{H})$, as they are distinct $k$-subsets of $[n]$, it is easy to see that $|T_1\cup T_2|\ge k+1$. Then we have
    $$\deg _{\mathcal{H}}(\{T_1,T_2\})=|\{\mathcal{J}'\in\mathcal{H}:\{T_1,T_2\}\subseteq\mathcal{J}'\}|\le n^{m-k-1},$$

    \noindent where the inequality holds since for fixed $T_1$ and $T_2$, we have at most $n^{m-k-1}$ choices for $V(\mathcal{J}')\setminus(T_1\cup T_2)$. Hence, by our choice of $\epsilon$, we have $\Delta_2(\mathcal{H})=\Theta(n^{m-k-1})\le d^{1-\ep}$.

    It remains to verify that $\mathcal{H}$ and $\mathcal{C}$ satisfy \cref{lem:usefule-lem} (C1)-(C4). We will need the following observation: For every edge $\{\mathcal{J}_1\ldots,\mathcal{J}_{\ell}\}\in \mathcal{C}$ and every $S\subsetneq [\ell]$ with size $1\le \ell'<\ell$, $\cup_{i\in [\ell]/S}V(\mathcal{J}_i)$ contains at most $(\ell-\ell')(m-k)-1$ vertices outside $\cup _{i\in S}V(\mathcal{J}_i)$. To see this, observe that
    \begin{equation}\label{eq:obs}
        \begin{aligned}
             \left|\bigcup_{i\in [\ell]/S}V(\mathcal{J}_i)\big/\bigcup_{j\in S}V(\mathcal{J}_j)\right|&= \left|\bigcup_{i\in [\ell]}V(\mathcal{J}_i)\big/\bigcup_{j\in S}V(\mathcal{J}_j)\right|\\
         &\le (\ell m-(\ell-1)k-1)-(\ell' m-(\ell'-1)k)\\
         &=(\ell-\ell')(m-k)-1,
    \end{aligned}
    \end{equation}

    \noindent where the inequality follows from the definition of $\mathcal{C}$ (see \eqref{eq:def-C}).
    For fixed $1\le \ell'< \ell$ and every subset $\mathcal{S}\subseteq V(\mathcal{C})=\mathcal{H}$ with size $\ell'$, we have
     \begin{equation}\label{eq:c1-c4}
         \begin{aligned}
         \deg_{\mathcal{C}^{(\ell)}}(\mathcal{S})=|\{\{\mathcal{J}_1\ldots,\mathcal{J}_{\ell}\}\in\mathcal{C}^{(\ell)}:\mathcal{S}\subseteq \{\mathcal{J}_1\ldots,\mathcal{J}_{\ell}\}\}|= O\left(n^{(\ell-\ell')(m-k)-1}\right),
     \end{aligned}
     \end{equation}

     \noindent where the second equality holds since by \eqref{eq:obs}, given $\mathcal{S}$, we have at most $n^{(\ell-\ell')(m-k)-1}$ choices for $\cup_{\mathcal{J}_i\in \{\mathcal{J}_1\ldots,\mathcal{J}_{\ell}\}/\mathcal{S}}V(\mathcal{J}_i)\setminus\cup_{\mathcal{J}_i\in \mathcal{S}}V(\mathcal{J}_j)$.

     Recall that $d =cn^{m-k}$. Setting $\ell'=1$ in \eqref{eq:c1-c4} yields that $$\Delta_1(\mathcal{C}^{(\ell)})=O\left(n^{(\ell-1)(m-k)-1}\right)=O\left(d^{\ell-1-\frac{1}{m-k}}\right)= O\left(d^{\ell-1}\right),$$ which verifies \cref{lem:usefule-lem} (C1). Similarly, by the choice of $\epsilon$, for all $2\le\ell'<\ell$ we have
     $$\Delta_{\ell'}(\mathcal{C}^{(\ell)})=O\left(n^{(\ell-\ell')(m-k)-1}\right)=O\left(d^{\ell-\ell'-\frac{1}{m-k}}\right)= O\left(d^{\ell-\ell'-\epsilon}\right),$$ which verifies \cref{lem:usefule-lem} (C2). Lastly, \cref{lem:usefule-lem} (C3), (C4) hold trivially as by \eqref{eq:obs}, $\Delta_1(\mathcal{C}^{(2)})=O(n^{m-k-1})=O(d^{1-\epsilon})$.

     Therefore, we have shown that $\mathcal{H}$ and $\mathcal{C}$ satisfy all the assumptions of \cref{lem:usefule-lem}. Consequently, for sufficiently large $n$, there exists a $\mathcal{C}$-free matching $\mathcal{M}\subseteq\mathcal{H}$ that covers at least $(1-d^{-\ep^3})|V(\mathcal{H})|\ge(1-3\epsilon)\binom{n}{k}$ vertices of $\mathcal{H}$, which implies that $|\mathcal{M}|\ge(1-3\epsilon)\binom{n}{k}/|\mathcal{J}|$. Combining \cref{claim:matching-packing} and the above lower bound on $|\mathcal{M}|$ conclude the proof of the lemma.
\end{proof}

\subsection{Proof of \cref{lem:H(F)}}

\noindent The following easy fact is needed for the proof of \cref{lem:H(F)}.

\begin{fact}\label{fact:pri of xor}
    Given positive integers $s$ and $n$, let $X_1,X_2,\cdots,X_s$ be distinct subsets of $[n]$. Let
    $$f(X_1,X_2,\cdots,X_s)=\sum_{i=1}^s|X_i|-|\bigcup_{i=1}^s X_i|.$$
    Then for $X_1',\ldots,X_s'\subseteq [n]$ with $X_i\subseteq X_i'$ for every $1\le i\le s$, we have $$f(X_1,X_2,\cdots,X_s)\le f(X_1',X_2',\cdots,X_s').$$
\end{fact}

\begin{proof}[Proof of \cref{lem:H(F)}]
    For fixed integers $\ell,e$ with $2\le\ell\le e$, let $F_1,\dots,F_{\ell}$ be arbitrary $\ell$ distinct edges in $\mathcal{H}(\mathcal{F})$, being contained in some $s$ distinct copies of $\mathcal{F}$, say,  $\mathcal{F}_1,\dots,\mathcal{F}_s$, where $s\ge 1$. For each $i\in [s]$, let $L_i \subseteq [\ell]$ be the set of $j$'s with $F_j\in\mathcal{F}_i$ and let $|L_i|=\ell_i$. Then $\{L_i:i\in [s]\}$ partitions $[\ell]$ and $\sum_{i=1}^s \ell_i=\ell$. For each $i\in[s]$, let $X_i = \cup_{j \in L_i} V(F_j)$. It is clear that $X_i\subseteq V(\mathcal{F}_i)$. Then, it follows from \cref{fact:pri of xor} and \cref{lem:H(F)} (i), (ii) that
    \begin{align*}
        \left|\bigcup_{i=1}^\ell F_i\right|&=\left|\bigcup_{i=1}^s X_i\right|=\sum_{i=1}^s |X_i|-f(X_1,\cdots,X_s)\\
        &\ge \sum_{i=1}^s |X_i|-f(V(\mathcal{F}_1),\cdots,V(\mathcal{F}_s))\\
        &=\sum_{i=1}^s|X_i|-\left(\sum_{i=1}^s|V(\mathcal{F}_i)|-\left|\bigcup_{i=1}^sV(\mathcal{F}_i)\right|\right)\\
        &\ge \sum_{i=1}^s\left(\ell_i\cdot tk-(\ell_i-1)k\right)-\left(sm-(sm-(s-1)k)\right)\\
        &=\ell\cdot tk-(\ell-1)k,
    \end{align*}
    as needed.
\end{proof}

\section{Proof of \cref{thm:main-canc}}\label{sec:thm:main-canc}

\subsection{The lower bound}\label{subsec:thm:main-canc-lower}

\noindent The goal of this subsection is to establish the lower bound of the limit in \cref{thm:main-canc}, i.e., for every fixed $\epsilon>0$ there exists an $n_0=n_0(t,k,\epsilon)$ such that for all $n\ge n_0$, we have $C_{2(t-1)}(n,tk)\ge(1-\epsilon)\cdot \binom{n}{k} / \binom{tk-1}{k-1}$.

The following lemma is the main technical result of this subsection. It shows that if $\mathcal{F}$ and $\mathcal{P}$ satisfy some nice properties, then the hypergraph $\mathcal{H}(\mathcal{F})$ defined in \cref{construction} is $2(t-1)$-cancellative.

\begin{lemma}\label{lem:canc-LB}
    Let $\mathcal{F}$ be a $2(t-1)$-cancellative $tk$-graph with $|V(\mathcal{F})|=m$ and $\mathcal{J}(\mathcal{F})$ be its $k$-shadow. Let $\mathcal{P}=\{(V(\mathcal{J}_i),\mathcal{J}_i):i\in[|\mathcal{P}|]\}$ be an induced $\mathcal{J}(\mathcal{F})$-packing in $\binom{[n]}{k}$. Suppose that $\mathcal{F}$ and $\mathcal{P}$ further satisfy the following properties:
    \begin{itemize}
        \item [{\rm (i)}] $\mathcal{F}$ is $\ell^-$-free (i.e., $(\ell\cdot tk-(\ell-1)k-1,\ell)$-free) for all $2\le\ell\le 2t$;
        \item [{\rm (ii)}] $\mathcal{V}=\{V(\mathcal{J}_i):i\in[|\mathcal{P}|]\}$ is $\ell^-$-free (i.e., $(\ell\cdot m-(\ell-1)k-1,\ell)$-free) for all $2\le\ell\le 2t$.
    \end{itemize}
    \noindent Then the $tk$-graph $\mathcal{H}(\mathcal{F})$ defined in \cref{construction} is $2(t-1)$-cancellative.
\end{lemma}

The following claim is very useful.

\begin{claim}\label{claim:induced}
    Let $\mathcal{F}$ be a $tk$-graph and $\mathcal{P}=\{(V(\mathcal{J}_i),\mathcal{J}_i):i\in[|\mathcal{P}|]\}$ be an induced $\mathcal{J}(\mathcal{F})$-packing. Let $\mathcal{H}(\mathcal{F})$ be defined in \cref{construction}. Suppose that $\mathcal{F}_i$ and $\mathcal{F}_{i'}$ are two distinct copies of $\mathcal{F}$ in $\mathcal{H}(\mathcal{F})$. Then for every $F\in\mathcal{F}_i$, we have $|F\cap V(\mathcal{F}_{i'})|\le k-1$. Moreover, if $\mathcal{F}$ is $2^-$-free, then for every distinct $F,F'\in\mathcal{H}(\mathcal{F})$, we have $|F\cap F'|\le k$.
\end{claim}

\begin{proof}
    To prove the first half of the claim, it is clear by \cref{construction} that $V(\mathcal{F}_{i})=V(\mathcal{J}_{i})$ and $V(\mathcal{F}_{i'})=V(\mathcal{J}_{i'})$. If $|F\cap V(\mathcal{F}_{i'})|\ge k$, then there exists a $k$-subset $T\subseteq F\subseteq V(\mathcal{J}_{i})$ such that $T\subseteq V(\mathcal{J}_{i})\cap V(\mathcal{J}_{i'})$ and $T\in\binom{F}{k}\subseteq \mathcal{J}(\mathcal{F}_{i})=\mathcal{J}_{i}$, contradicting the assumption that $\mathcal{P}$ is induced.

    To prove the second half of the claim, just note that if $F,F'$ are contained in the same copy of $\mathcal{F}$, then $|F\cap F'|\le k$ is a direct consequence of the $2^-$-freeness of $\mathcal{F}$. If $F,F'$ are contained in distinct copies of $\mathcal{F}$, say $F\in\mathcal{F}_{i}$ and $F'\in\mathcal{F}_{i'}$, where $i\neq i'$, then according to the previous discussion, we have $|F\cap F'|\le|F\cap V(\mathcal{F}_{i'})|\le k-1$, as needed.
\end{proof}

Now we are ready to present the proof of \cref{lem:canc-LB}.

\begin{proof}[Proof of \cref{lem:canc-LB}]
    First of all, applying \cref{lem:H(F)} with $e=2t$ one can easily see that $\mathcal{H}(\mathcal{F})$ is $\ell^-$-free for all $2\le\ell\le 2t$. Suppose for the sake of contradiction that $\mathcal{H}(\mathcal{F})$ is not $2(t-1)$-cancellative. Then there exist $2t$ distinct edges, say $F_1,\ldots,F_{2t-2},F,F'\in\mathcal{H}(\mathcal{F})$, such that \begin{align}\label{eq:cancellative-not}
        F\Delta F'\subseteq\bigcup_{i=1}^{2t-2}F_i.
    \end{align}

    \paragraph{Case 1.} Suppose that $F$ and $F'$ belong to the same copy, say $\mathcal{F}_i$, of $\mathcal{F}$. Since $\mathcal{F}$ is $2^-$-free, we have $|F\cap F'|\le k$. Let $S\subseteq [2t-2]$ be the set of $j$'s with $F_j\in \mathcal{F}_i$ and let $|S|=s$. Since $\mathcal{F}$ is $2(t-1)$-cancellative, we have $0\le s<2(t-1)$. Moreover, it follows from \cref{claim:induced} that for each $j\in [2t-2]\setminus S$,
    \begin{align}\label{eq:induced}
        |F_j\cap (F\Delta F')|\le |F_j\cap V(\mathcal{F}_i)|\le k-1.
    \end{align}

    \noindent Consequently, we have
    \begin{align*}
        \left|F\cup F'\cup \left(\bigcup_{j\in S} F_j\right)\right|= &\left|\bigcup_{j\in S} F_j\right|+\left|(F\cup F')\big/\bigcup_{j\in S} F_j\right|\\
        \le&\left|\bigcup_{j\in S} F_j\right|+|F\cap F'|+|F\Delta F'|-\left|(F\Delta F')\cap\left(\bigcup_{j\in S}F_j\right)\right|\\
        \le&\left|\bigcup_{j\in S} F_j\right|+|F\cap F'|+\left|(F\Delta F')\cap\left(\bigcup_{j\in [2t-2]\setminus S}F_j\right)\right|\\
        \le&s\cdot tk+k+(2t-2-s)\cdot (k-1)\\
        =&(s+2)\cdot tk-(s+1)k-(2t-2-s)\\
        \le&(s+2)\cdot tk-(s+1)k-1,
    \end{align*}

    \noindent where the second inequality follows from \eqref{eq:cancellative-not} and the third inequality follows from \eqref{eq:induced}. This implies that $F,F'$ and $\{F_j:j\in S\}$ form a $(s+2)^-$-configuration, where $2\le s+2< 2t$, which contradicts the fact that $\mathcal{H}(\mathcal{F})$ is $\ell^-$-free for all $2\le\ell\le 2t$.

   \paragraph{Case 2.} Suppose that $F$ and $F'$ belong to different copies of $\mathcal{F}$, say $F\in\mathcal{F}_i$ and $F'\in\mathcal{F}_{i'}$, $i\ne i'$. Then, it follows from \cref{claim:induced} that $|F\cap F'|\le|F\cap V(\mathcal{F}_{i'})|\le k-1$. Therefore, we have
   \begin{align*}
       \left|(F\cup F')\cup\left(\bigcup_{j=1}^{2t-2}F_j\right)\right|=&\left|\bigcup_{j=1}^{2t-2}F_j\right|+\left|(F\cup F')/\bigcup_{j=1}^{2t-2}F_j\right|\le\left|\bigcup_{j=1}^{2t-2}F_j\right|+|F\cap F'|\\
       \le&(2t-2)\cdot tk+(k-1)=2t\cdot tk-(2t-1)k-1,
   \end{align*}
   where the first inequality follows from \eqref{eq:cancellative-not}. This implies that $F,F'$ and $\{F_j:1\le j\le 2t-2\}$ form a $2t^-$-configuration, which contradicts the fact that $\mathcal{H}(\mathcal{F})$ is $2t^-$-free.
\end{proof}

Next, we will construct a $tk$-graph $\mathcal{F}$ that satisfies the assumptions of \cref{lem:canc-LB} and has large enough ratio $\frac{|\mathcal{F}|}{|\mathcal{J}(\mathcal{F})|}$. Let $m_0$ be an integer that will be determined later. Let $\mathcal{G}=\{G_i:i\in[|\mathcal{G}|]\}\subseteq\binom{[m_0]}{tk-1}$ be a $(tk-1)$-graph obtained by applying \cref{lem:e^--free} with $r=tk-1$ and $e=2t$. Let $U=\{u_i:i\in[|\mathcal{G}|]\}$ be a set of $|\mathcal{G}|$ distinct vertices disjoint from $V(\mathcal{G})=[m_0]$. Let $\mathcal{F}$ be the $tk$-graph formed by pairing up the edges in $\mathcal{G}$ and the vertices in $U$, that is,
\begin{align}\label{eq:canc-F}
    \mathcal{F}=\left\{G_i\cup \{u_{i}\}:i\in[|\mathcal{G}|]\right\}.
\end{align}

\noindent Then
\begin{align}\label{eq:F}
    |\mathcal{F}|=|\mathcal{G}|\ge cm_0^{k+\fr{1}{2t-1}},
\end{align}

\noindent where $c=c(tk-1,k,2t)$ is a constant independent of $m_0$. Pick $m_0=m_0(t,k)\ge(2c^{-1}\epsilon^{-1})^{2t-1}$ so that
\begin{align}\label{eq:m}
    \frac{1}{c^{-1}m_0^{-\frac{1}{2t-1}}+\binom{tk-1}{k-1}}\ge(1-\epsilon/2)\cdot\frac{1}{\binom{tk-1}{k-1}}.
\end{align}

We have the following lemma on the properties of $\mathcal{F}$.

\begin{lemma}\label{lem:canc-LB-F}
Let $\mathcal{F}$ be the $tk$-graph defined as above. Then
\begin{itemize}
    \item [{\rm (i)}] $\mathcal{F}$ is $2(t-1)$-cancellative and $\ell^-$-free for all $2\le\ell\le 2t$;
    \item [{\rm (ii)}] $\frac{|\mathcal{F}|}{|\mathcal{J}(\mathcal{F})|}\ge (1-\epsilon/2)\cdot\frac{1}{\binom{tk-1}{k-1}}$.
\end{itemize}
\end{lemma}

\begin{proof}
    $\mathcal{F}$ is clearly $2(t-1)$-cancellative since by definition, for every $i$, $u_i$ appears only in $G_i\cup \{u_{i}\}\in\mathcal{F}$ but not in any other $G_j\cup \{u_j\}\in\mathcal{F}$ for $i\neq j$. To show that $\mathcal{F}$ is $\ell^-$-free for all $2\le\ell\le 2t$, it suffices to notice that for any $\ell$ distinct edges of $\mathcal{F}$, say $\{G_i \cup \{u_{i}\} : i \in [\ell]\} \subseteq \mathcal{F}$, we have $$\left|\bigcup_{i=1}^{\ell} \left(G_i \cup \{u_{i}\}\right)\right| = \left|G_1 \cup \cdots \cup G_{\ell}\right| + \ell \geq \ell\cdot(tk - 1) - (\ell - 1)k + \ell = \ell\cdot tk - (\ell - 1)k,$$
    where the inequality follows from the $\ell^-$-freeness of $\mathcal{G}$. Lastly, the lower bound on $\frac{|\mathcal{F}|}{|\mathcal{J}(\mathcal{F})|}$ follows from \eqref{eq:F}, \eqref{eq:m}, and the observation that $|\mathcal{J}(\mathcal{F})|\le\binom{m_0}{k}+|\mathcal{F}|\cdot\binom{tk-1}{k-1}$.
\end{proof}

Putting Lemmas \ref{lem:canc-LB} and \ref{lem:canc-LB-F} together, next we prove the lower bound of \cref{thm:main-canc}.

\begin{proof}[Proof of \cref{thm:main-canc}, the lower bound]
    Let $\mathcal{F}\subseteq\binom{[m]}{tk}$ be defined as in \eqref{eq:canc-F}, where $m=m_0+|\mathcal{G}|$. Then $\mathcal{F}$ satisfies the conclusions of \cref{lem:canc-LB-F}. Let $\mathcal{J}(\mathcal{F})\subseteq\binom{[m]}{k}$ be the $k$-shadow of $\mathcal{F}$. Applying \cref{lem:conflict-free-matching} with $\mathcal{J}=\mathcal{J}(\mathcal{F})$ and $e=2t$, we conclude that for every $\epsilon>0$, there is an integer $n_0=n_0(m,\epsilon)=n_0(t,k,\epsilon)$ such that for all $n\ge n_0$, there is an induced $\mathcal{J}(\mathcal{F})$-packing $\mathcal{P}$ in $\binom{[n]}{k}$ that is $(\ell\cdot m-(\ell-1)k-1,\ell)$-free for all $2\le\ell\le 2t$, and moreover,
    \begin{align}\label{eq:|P|}
        |\mathcal{P}|\ge(1-\epsilon/2)\cdot\binom{n}{k}\big/|\mathcal{J}(\mathcal{F})|.
    \end{align}

    \noindent Clearly, $\mathcal{F}$ and $\mathcal{P}$ satisfy the assumptions of \cref{lem:canc-LB}. It implies that the $tk$-graph $\mathcal{H}(\mathcal{F})$ defined in \cref{construction} is $2(t-1)$-cancellative. Moreover, it follows from \cref{lem:canc-LB-F} (ii), \eqref{eq:|P|}, and \eqref{eq:H-sparse1} that
    \begin{align*}
        |\mathcal{H}(\mathcal{F})|=|\mathcal{P}|\cdot|\mathcal{F}|\ge(1-\epsilon)\cdot\binom{n}{k} \bigg/ \binom{tk-1}{k-1},
    \end{align*}

    \noindent as needed.
\end{proof}

\subsection{The upper bound}

\noindent The goal of this subsection is to prove the upper bound $C_{2(t-1)}(n,tk)\le\binom{n}{k} / \binom{tk-1}{k-1}$. Recall that for a hypergraph $\mathcal{F}$ and $T\subseteq V(\mathcal{F})$, we use $\deg_{\mathcal{F}}(T)$ to denote the number of edges in $\mathcal{F}$ that contain $T$. For a subset $X\subseteq V(\mathcal{F})$ and an integer $s\ge 0$, we denote $\mathcal{D}(\mathcal{F},X,s):=\{T\in\binom{X}{k}:\deg_{\mathcal{F}}(T)=s\}$ and $\mathcal{D}(\mathcal{F},X,\ge s):=\{T\in\binom{X}{k}:\deg_{\mathcal{F}}(T)\ge s\}$. If $X=V(\mathcal{F})$, we denote $\mathcal{D}(\mathcal{F},V(\mathcal{F}),s)=\mathcal{D}(\mathcal{F},s)$ and $\mathcal{D}(\mathcal{F},V(\mathcal{F}),\ge s)=\mathcal{D}(\mathcal{F},\ge s)$ for simplicity.

Let $\mathcal{F}\subseteq\binom{[n]}{tk}$ be a $2(t-1)$-cancellative $tk$-graph, we aim to show that
\begin{align}\label{eq:canc-upper}
    \binom{tk}{k}|\mathcal{F}|\le t\binom{n}{k}.
\end{align}

\noindent To that end, observe that
\begin{align}\label{eq:canc-upper-1}
    \binom{tk}{k}|\mathcal{F}|=\sum_{s\ge1} s\cdot|\mathcal{D}(\mathcal{F},s)|=\sum_{s\ge1}|\mathcal{D}(\mathcal{F},s)|+\sum_{s\ge2}(s-1)\cdot|\mathcal{D}(\mathcal{F},s)|,
\end{align}

\noindent and moreover,
\begin{align}\label{eq:canc-upper-2}
    \sum_{s\ge1}|\mathcal{D}(\mathcal{F},s)|\le \sum_{s\ge0}|\mathcal{D}(\mathcal{F},s)|=|\mathcal{D}(\mathcal{F},\ge 0)|=\binom{n}{k}.
\end{align}

\noindent Combining \eqref{eq:canc-upper-1} and \eqref{eq:canc-upper-2}, one can infer that to prove \eqref{eq:canc-upper}, it suffices to show that
\begin{align}\label{eq:canc-upper-3}
    \sum_{s\ge2}(s-1)\cdot|\mathcal{D}(\mathcal{F},s)|\le (t-1)\binom{n}{k}.
\end{align}

To prove \eqref{eq:canc-upper-3}, we will apply a result on the maximum number of edges that can be contained in a hypergraph with bounded matching number. Recall that a matching in a hypergraph $\mathcal{F}$ is a set of pairwise disjoint edges in $\mathcal{F}$. The {\it matching number} of $\mathcal{F}$, denoted by $\nu(\mathcal{F})$, is the largest cardinality of a matching in $\mathcal{F}$. Let $m(n,k,t)$ be the maximum number of edges that can be contained in an $n$-vertex $k$-graph with matching number at most $t$. Erd\H{o}s \cite{erdos1965problem} posed a well-known conjecture, commonly termed the Erd\H{o}s Matching conjecture, on the exact value of $m(n,k,t)$. This conjecture is still open. The interested read is referred to \cite{Frankl-Kupavskii-The-Erdos-Matching-Con} for its recent progress. For our purpose, we will make use of the following result of Frankl \cite{FRANKL80} and Kleitman \cite{kleitman1968maximal} on the special case of this conjecture.

\begin{lemma}[\cite{FRANKL80,kleitman1968maximal}, see also \cite{Frankl-Kupavskii-The-Erdos-Matching-Con}]\label{lem:matching-conjecture}
    $m((t-1)k,k,t-2)=\binom{(t-1)k-1}{k}$ for all positive integers $k\ge2,t\ge 2$.
\end{lemma}

With \cref{lem:matching-conjecture} at hand, we are ready to prove \eqref{eq:canc-upper-3}.

\begin{proof}[Proof of \cref{thm:main-union-free}, the upper bound]
    Let $\mathcal{F}\subseteq\binom{n}{tk}$ be a $2(t-1)$-cancellative $tk$-graph. According to the discussions above, to prove $|\mathcal{F}|\le\binom{n}{k} / \binom{tk-1}{k-1}$, it is enough to prove \eqref{eq:canc-upper-3}.

    Fix $s\ge 2$ and $T\in\mathcal{D}(\mathcal{F},s)$. Let $\{F_1,\ldots,F_s\}$ be the set of edges in $\mathcal{F}$ that contain $T$. For each $i\in[s]$, denote $\mathcal{D}(\mathcal{F},F_i\setminus T,\ge 2)=\{R\in\binom{F_i\setminus T}{k}:\deg_{\mathcal{F}}(R)\ge 2\}$. We have the following claim.

    \begin{claim}\label{claim:uf-UB}
        There are at least $s-1$ choices of $i\in[s]$ such that $\nu(\mathcal{D}(\mathcal{F},F_i\setminus T,\ge 2))\le t-2$.
    \end{claim}

    To prove the claim, suppose for contradiction that there are at least two choices of $i\in[s]$, say $i\in\{1,2\}$ such that
    \begin{align}\label{eq:matching-t-1}
        \nu(\mathcal{D}(\mathcal{F},F_i\setminus T,\ge 2))\ge t-1
    \end{align}

    \noindent As $|F_1|=tk$ and $|F_1\setminus T|=(t-1)k$, \eqref{eq:matching-t-1} implies that there exist $T_{1,1},\ldots,T_{1,t-1}\in\mathcal{D}(\mathcal{F},F_1\setminus T,\ge 2)\subseteq\mathcal{D}(\mathcal{F},\ge 2)$ such that $F_1$ can be written as the disjoint union $F_1=T\cup T_{1,1}\cup\cdots\cup T_{1,t-1}$. As $T_{1,j}\in\mathcal{D}(\mathcal{F},\ge 2)$ for each $j\in [t-1]$, there exist not necessarily distinct $F_{1,1},\ldots,F_{1,t-1}\in\mathcal{F}\setminus\{F_1\}$ such that $T_{1,j}\subseteq F_{1,j}$ for each $j\in [t-1]$. It follows that
    \begin{align}\label{eq:contradiction-1}
        F_1\subseteq F_2\cup F_{1,1}\cup\cdots\cup F_{1,t-1}.
    \end{align}

    \noindent Similarly, we can show that there exist not necessarily distinct $F_{2,1},\ldots,F_{2,t-1}\in\mathcal{F}\setminus\{F_2\}$ such that
    \begin{align}\label{eq:contradiction-2}
        F_2\subseteq F_1\cup F_{2,1}\cup\cdots\cup F_{2,t-1}.
    \end{align}

    \noindent Note that it is possible that $F_2\in\{F_{1,1},\ldots,F_{1,t-1}\}$ and $F_1\in\{F_{2,1},\ldots,F_{2,t-1}\}$. In any case, given \eqref{eq:contradiction-1} and \eqref{eq:contradiction-2}, we can infer that the symmetric difference $F_1\Delta F_2$ is contained at most $2t-2$ edges in $\mathcal{F}\setminus\{F_1,F_2\}$, a contradiction. This completes the proof of the claim.

    Let $S=\{i\in[s]:\nu(\mathcal{D}(\mathcal{F},F_i\setminus T,\ge 2))\le t-2\}$. By \cref{claim:uf-UB}, we have $|S|\ge s-1$. For each $i\in S$, it follows from \cref{lem:matching-conjecture} that
    $$|\mathcal{D}(\mathcal{F},F_i\setminus T,\ge 2)|\le m((t-1)k,k,t-2)=\binom{(t-1)k-1}{k},$$
    which further implies that for each $i\in S$,
    \begin{align}\label{eq:own-k-subsets}
        |\mathcal{D}(\mathcal{F},F_i\setminus T,1)|\ge\binom{(t-1)k}{k}-\binom{(t-1)k-1}{k}=\binom{(t-1)k-1}{k-1}.
    \end{align}

    Define a map $$\sigma:\mathcal{D}(\mathcal{F},\ge 2)\to 2^{\mathcal{D}(\mathcal{F},1)}$$ that maps each $T\in\mathcal{D}(\mathcal{F},\ge 2)$ to
    $$\sigma(T):=\{T'\in\mathcal{D}(\mathcal{F},1):\exists F\in \mathcal{F} \text{ such that } T\subseteq F \text{ and } T'\subseteq F/T \}.$$
    It follows from \eqref{eq:own-k-subsets} that for each $T\in\mathcal{D}(\mathcal{F},\ge 2)$,
    \begin{align}\label{eq:sigma-T}
        |\sigma(T)|=\sum_{T\subseteq F\in\mathcal{F}}|\mathcal{D}(\mathcal{F},F\setminus T,1)|\ge(\deg_{\mathcal{F}}(T)-1)\cdot \binom{(t-1)k-1}{k-1}.
    \end{align}

    Let us count the number of pairs $N:=|\{(T,T'):T\in\mathcal{D}(\mathcal{F},\ge 2),T'\in\sigma(T)\}|$ in two ways. On one hand, by \eqref{eq:sigma-T} we have
    \begin{align}\label{eq:N-lower-bound}
        N=\sum_{T\in\mathcal{D}(\mathcal{F},\ge 2)}|\sigma(T)|=\sum_{s\ge2}\sum_{T\in\mathcal{D}(\mathcal{F},s)}|\sigma(T)|\ge \sum_{s\ge2}|\mathcal{D}(\mathcal{F},s)|(s-1)\binom{(t-1)k-1}{k-1}.
    \end{align}

    \noindent On the other hand,
    \begin{equation}\label{eq:N-upper-bound}
        \begin{aligned}
        N&=\sum_{T'\in \mathcal{D}(\mathcal{F},1)}|\{T\in \mathcal{D}(\mathcal{F},\ge 2): \exists F\in \mathcal{F} \text{ such that } T\subseteq F \text{ and } T'\subseteq F/T \}|\\
         &\le |\mathcal{D}(\mathcal{F},1)|\binom{(t-1)k}{k}.
    \end{aligned}
    \end{equation}

    \noindent Combining \eqref{eq:N-lower-bound} and \eqref{eq:N-upper-bound}, we have
    $$\sum_{s\ge2}|\mathcal{D}(\mathcal{F},s)|(s-1)\binom{(t-1)k-1}{k-1}\le |\mathcal{D}(\mathcal{F},1)|\binom{(t-1)k}{k}.$$
    Rearranging gives that
    $$\sum_{s\ge2}(s-1)\cdot|\mathcal{D}(\mathcal{F},s)|\le (t-1)|\mathcal{D}(\mathcal{F},1)|\le (t-1)\binom{n}{k},$$
    completing the proof of \eqref{eq:canc-upper-3}.
\end{proof}

\section{Proof of \cref{thm:main-union-free}}\label{sec:thm:main-union-free}

\noindent In this section, we aim to prove \cref{thm:main-union-free}. Given the upper bound in \eqref{eq:ST-union-free}, it suffices to prove the lower bound of the limit in \cref{thm:main-union-free}, i.e., for every fixed $\epsilon>0$ there exists an $n_0=n_0(t,k,\epsilon)$ such that for all $n\ge n_0$, we have $U_{t+1}(n,tk)\ge(1-\epsilon)\cdot \binom{n}{k} / \binom{tk-1}{k-1}$.

To do so, we roughly adapt the same approach as in \cref{subsec:thm:main-canc-lower}. We first prove a lemma that shows that if $\mathcal{F}$ and $\mathcal{P}$ satisfy several required properties, then the hypergraph $\mathcal{H}(\mathcal{F})$ defined in \cref{construction} is $(t+1)$-union-free. Then we show the existence of such $\mathcal{F}$ and $\mathcal{P}$.

\begin{lemma}\label{lem:uf-LB}
    Let $\mathcal{F}$ be a $(t+1)$-union-free $tk$-graph with $|V(\mathcal{F})|=m$ and $\mathcal{J}(\mathcal{F})$ be its $k$-shadow. Let $\mathcal{P}=\{(V(\mathcal{J}_i),\mathcal{J}_i):i\in[|\mathcal{P}|]\}$ be an induced $\mathcal{J}(\mathcal{F})$-packing in $\binom{[n]}{k}$. Suppose that $\mathcal{F}$ and $\mathcal{P}$ further satisfy the following properties:
    \begin{itemize}
        \item [{\rm (i)}] $\mathcal{F}$ is $\ell^-$-free (i.e., $(\ell\cdot tk-(\ell-1)k-1,\ell)$-free) for all $2\le\ell\le 2t+2$;
        \item [{\rm (ii)}] $\mathcal{V}=\{V(\mathcal{J}_i):i\in[|\mathcal{P}|]\}$ is $\ell^-$-free (i.e., $(\ell\cdot m-(\ell-1)k-1,\ell)$-free) for all $2\le\ell\le 2t+2$.
    \end{itemize}
    \noindent Then the $tk$-graph $\mathcal{H}(\mathcal{F})$ defined in \cref{construction} is $t$-cover-free and $(t+1)$-union-free.
\end{lemma}

\begin{proof}
    First of all, applying \cref{lem:H(F)} with $e=2t+2$ one can easily see that $\mathcal{H}(\mathcal{F})$ is $\ell^-$-free for all $2\le\ell\le 2t+2$.

    We proceed to show that $\mathcal{H}(\mathcal{F})$ is $t$-cover-free. If not, then there exist $t+1$ distinct edges, say $F_1,\ldots,F_{t+1}\in\mathcal{H}(\mathcal{F})$, such that $F_{t+1}\subseteq \cup_{i=1}^{t}F_i$. This implies that
    $$\left|\bigcup_{i=1}^{t+1}F_i\right|=\left|\bigcup_{i=1}^{t}F_i\right|\le t\cdot tk.$$

    \noindent If $|\cup_{i=1}^{t+1}F_i|< t\cdot tk=(t+1)\cdot tk-tk$, then we obtain a contradiction as $\mathcal{H}(\mathcal{F})$ is $(t+1)^-$-free.
    If $|\cup_{i=1}^{t}F_i|=t\cdot tk$, then $F_1,\ldots,F_t$ must be pairwise disjoint. Moreover, it follows from \cref{claim:induced} that $|F_{t+1}\cap F_i|\le k$ for all $i\in [t]$. Since $F_{t+1}\subseteq \cup_{i=1}^{t}F_i$, it is not hard to check that $|F_{t+1}\cap F_i|= k$ for all $i\in [t]$. It follows again from \cref{claim:induced} that $F_1,\ldots,F_{t+1}$ must be all contained in the same copy of $\mathcal{F}$, contradicting the assumption that $\mathcal{F}$ is $t$-cover-free.

    We proceed to show that $\mathcal{H}(\mathcal{F})$ is $(t+1)$-union-free. If not, then there exist distinct $\mathcal{A},\mathcal{B}\subseteq \mathcal{H}(\mathcal{F})$, with $1\le |\mathcal{A}|,|\mathcal{B}|\le t+1$, such that $\cup_{A\in \mathcal{A}}A=\cup_{B\in \mathcal{B}}B$. Since $\mathcal{H}(\mathcal{F})$ is $t$-cover-free, it is not hard to see that $\cup_{A\in \mathcal{A}}A=\cup_{B\in \mathcal{B}}B$ is only possible when $|\mathcal{A}|=|\mathcal{B}|=t+1$.
    Assume without loss of generality that
    $$\mathcal{A}=\{C_1,\ldots,C_p,A_{p+1},\ldots, A_{t+1}\},\mathcal{B}=\{C_1,\ldots,C_p,B_{p+1},\ldots,B_{t+1}\},$$
    where $\mathcal{A}\cap\mathcal{B}=\{C_1,\ldots,C_p\}$ and $0\le p\le t$. Denote $X=\cup_{i=1}^pC_i$, $Y=\cup_{i=p+1}^{t+1}A_i$, and $Z=\cup_{i=p+1}^{t+1}B_i$. As $X\cup Y=X\cup Z$, we have
    \begin{align}\label{eq-union-free}
        Y\Delta Z\subseteq X.
    \end{align}

    \noindent We now consider the size of $X\cup Y\cup Z$. By \eqref{eq-union-free}, we have
    \begin{align}\label{eq-union-free-upper}
        |X\cup Y\cup Z|=|X\cup Y|=|X\cup Z|\le (t+1)\cdot tk.
    \end{align}

    \noindent Note that $\mathcal{A}\cup\mathcal{B}$ consists of exactly $2t+2-p$ edges of $\mathcal{H}(\mathcal{F})$. Since $\mathcal{H}(\mathcal{F})$ is $(2t+2-p)^-$-free, we have
    \begin{align}\label{eq-union-free-lower}
        |X\cup Y\cup Z|\ge (2t+2-p)\cdot tk-(2t+2-p-1)k.
    \end{align}

    \noindent Comparing the upper and lower bounds of $|X\cup Y\cup Z|$ in \eqref{eq-union-free-upper} and \eqref{eq-union-free-lower}, it is not hard to observe that
    $$(t+1)\cdot tk\ge (2t+2-p)\cdot tk-(2t+2-p-1)k$$
    \noindent only if $t-\frac{1}{t-1}\le p\le t$, which happens when $t\ge 2$ and $p=t$ or $t=2$ and $p=1$.

    In the sequel, we will discuss these two cases separately.

    \paragraph{Case 1: $t\ge 2$ and $p=t$.} In this case, we have $Y=A_{t+1}$ and $Z=B_{t+1}$. By \cref{claim:induced}, we can always assume that $|A_{t+1}\cap B_{t+1}|=k-x$, where $x\ge 0$. Given that $Y\Delta Z=A_{t+1}\Delta B_{t+1}\subseteq X$, we have
    $$|X\cup A_{t+1}\cup B_{t+1}|\le |X|+|A_{t+1}\cap B_{t+1}|\le t\cdot tk + k-x.$$
    Noting that $\mathcal{H}(\mathcal{F})$ is $(t+2)^-$-free, we also have
    $$|X\cup A_{t+1}\cup B_{t+1}|\ge(t+2)\cdot tk-(t+1)k.$$
    Combining the two inequalities above, it is not hard to observe that
    $$t\cdot tk + k-x\ge (t+2)\cdot tk-(t+1)k$$
    only if $t=2$ and $x=0$. Now, $\mathcal{H}(\mathcal{F})$ is a $2k$-graph. Moreover, we have $C_1\cup C_2\cup A_3=C_1\cup C_2\cup B_3$ and $|A_3\cap B_3|=k$. It follows from \cref{claim:induced} that $A_3$ and $B_3$ belong to the same copy of $\mathcal{F}$, say $\mathcal{F}_i$. Since $\mathcal{F}_i$ is $3$-union-free, at least one of $C_1$ and $C_2$ does not belong to  $\mathcal{F}_i$. Assume without loss of generality that $C_1\notin \mathcal{F}_i$. Again, by \cref{claim:induced} we can infer that $|C_1\cap (A_3\Delta B_3)|\le |C_1\cap V(\mathcal{F}_i)|\le k-1$. As $A_3\Delta B_3\subseteq X=C_1\cup C_2$ and $|A_3\Delta B_3|=2k$, we have
    $$|C_2\cap (A_3\Delta B_3)|=|A_3\Delta B_3|-|C_1\cap (A_3\Delta B_3)|\ge 2k-(k-1)=k+1.$$
    Therefore, $$|C_2\cup A_3\cup B_3|\le|C_2|+|A_3\cup B_3|-|C_2\cap (A_3\Delta B_3)|\le 4k-1=3\cdot 2k-2k-1,$$
    which contradicts the fact that $\mathcal{H}(\mathcal{F})$ is $3^-$-free.

    \paragraph{Case 2: $t=2$ and $p=1$.} In this case, we have $C_1\cup A_2\cup A_3=C_1\cup B_2\cup B_3$. It follows by \eqref{eq-union-free-upper} that
    $$|C_1\cup A_2\cup A_3\cup B_2\cup B_3|\le 6k.$$
    On the other hand, since $\mathcal{H}(\mathcal{F})$ is $5^-$-free, we have
    $$|C_1\cup A_2\cup A_3\cup B_2\cup B_3|\ge 5\cdot 2k-4k=6k.$$
    Consequently,
    $$|C_1\cup A_2\cup A_3\cup B_2\cup B_3|=|C_1\cup A_2\cup A_3|=|C_1\cup B_2\cup B_3|=6k.$$
    Since $\mathcal{H}(\mathcal{F})$ is a $2k$-graph, the above equation implies that $C_1,A_2,A_3$ are pairwise disjoint, and so are $C_1,B_2,B_3$. This further implies that $A_2\cup A_3=B_2\cup B_3$. Therefore, we conclude that
    $$|A_2\cup A_3\cup B_2\cup B_3|=|A_2\cup A_3|\le 4k,$$
    which contradicts the fact that $\mathcal{H}(\mathcal{F})$ is $4^-$-free. This completes the proof.
\end{proof}

To construct $\mathcal{F}$ that satisfies the assumptions of \cref{lem:uf-LB}, we apply the same construction as in \eqref{eq:F}. The only difference is that, when constructing $\mathcal{G}\subseteq\binom{[m_0]}{tk-1}$, we apply \cref{lem:e^--free} with $e=2t+2$. This yields a $tk$-graph $\mathcal{F}\subseteq\binom{[m]}{tk}$ with $|\mathcal{F}|=|\mathcal{G}|\ge c'm_0^{k+\fr{1}{2t+1}}$,  where $m=m_0+|\mathcal{G}|$ and $c'=c'(tk-1,k,2t+2)$. Pick $m_0=m_0(t,k)$ so that
\begin{align*}
    \frac{1}{\frac{1}{c'}m_0^{-\frac{1}{2t+1}}+\binom{tk-1}{k-1}}\ge(1-\epsilon/2)\cdot\frac{1}{\binom{tk-1}{k-1}}.
\end{align*}

Similarly to \cref{lem:canc-LB-F}, we can prove the following lemma, whose proof is omitted.

\begin{lemma}\label{lem:uf-LB-F}
Let $\mathcal{F}$ be the $tk$-graph defined as above. Then
\begin{itemize}
    \item [{\rm (i)}] $\mathcal{F}$ is $(t+1)$-union-free and $\ell^-$-free for all $2\le\ell\le 2t$;
    \item [{\rm (ii)}] $\frac{|\mathcal{F}|}{|\mathcal{J}(\mathcal{F})|}\ge (1-\epsilon/2)\cdot\frac{1}{\binom{tk-1}{k-1}}$.
\end{itemize}
\end{lemma}

\begin{proof}[Proof of \cref{thm:main-union-free}]
    As mentioned in the beginning of this section, to prove \cref{thm:main-union-free}, it suffices to prove $U_{t+1}(n,tk)\ge(1-o(1))\cdot \binom{n}{k} / \binom{tk-1}{k-1}$. Indeed, the proof is exactly the same with $C_{2(t-1)}(n,tk)\ge(1-o(1))\cdot \binom{n}{k} / \binom{tk-1}{k-1}$. The only difference is that, in the previous case we apply \cref{lem:conflict-free-matching} with $e=2t$, and here we use $e=2t+2$. We omit the details.
\end{proof}

\section{Concluding remarks}\label{sec:conclusion}

\noindent In this paper, we obtain asymptotically sharp bounds on $C_{2(t-1)}(n,tk)$ and $U_{t+1}(n,tk)$ for all $t\ge 2$ and $k\ge 2$. However, the determination of $C_t(n,r)$ and $U_t(n,r)$ for general $t$ and $r$ remains widely open. There are still huge gaps between the general lower and upper bounds on the Tur\'an exponents of these two functions (see \eqref{eq:ST-canc-bound}, \eqref{eq:cff-bound}, \eqref{eq:Shangguan-Tamo-union-free}). It would be very interesting to close or narrow these gaps.

We will end this paper by discussing some open cases about $C_t(n,r)$ for $t\in\{1,2,3\}$. It is known that $C_1(n,r)=\Theta(n^r)$, and by \cite{Katona-Nemetz-Simonovits-Turan-density}, the Tur\'an density $\lim_{n\rightarrow\infty}\frac{C_1(n,r)}{\binom{n}{r}}$ always exists. Until now, its exact value is still unknown for all $r\ge 5$ (see \cite{bollobas1974three,sidorenko1987maximal} for $r\in\{3,4\}$). We have determined the asymptotic value of $C_2(n,r)$ for all even $r$. For odd $r$, \cite{Furedi-2-canc,Shangguan-Tamo-canc-and-union-free} showed that $n^{k-o(1)}<C_2(n,2k-1)=O(n^{k})$ for all $k\ge 2$. F\"uredi \cite{Furedi-2-canc}  conjectured that $C_2(n,2k-1)=o(n^{k})$; but this conjecture may be notoriously hard as he also showed that $C_2(n,3)$ has the same order of magnitude as the extremal function of the famous $(7,4)$-problem for $3$-graphs (see Theorem 4.1 in \cite{Furedi-2-canc}). For $t=3$, one can show that $\Omega(n^{\frac{5}{4}})=C_3(n,3)=o(n^2)$, $\Omega(n^{\frac{7}{4}})=C_3(n,4)=O(n^2)$, and $\Omega(n^{2})=C_3(n,5)=o(n^3)$, where the upper bounds on $C_3(n,3)$ and $C_3(n,5)$ follow from a standard application of the hypergraph removal lemma (see \cite{Conlon-Fox-rem-lem-survey} for example), and the other bounds follow from \eqref{eq:ST-canc-bound}.

\section*{Acknowledgments}
\noindent This project is supported by the National Natural Science Foundation of China under Grant Nos. 12101364 and 12231014, and the Natural Science Foundation of Shandong Province under Grant No. ZR2021QA005.

{\small
\normalem
\bibliographystyle{plain}
\bibliography{mid-autumn}}

\appendix
\section{Proof of \cref{claim-concentration}}\label{appendix-concentration}

\noindent Before proving \cref{claim-concentration}, we first introduce the bounded differences inequality, a concentration inequality which bounds the deviation between the sampled value and the expected value of certain functions when they are evaluated on independent random variables.
\begin{lemma}[see Theorem 9.1.3 in \cite{zhaoyufei22prob}]\label{lem:BDI}
    Let $\Omega_1,\ldots,\Omega_s$ be finite subsets and $X_1\in\Omega_1,\ldots,X_s\in\Omega_s$ be independent random variables. Suppose that $f:\Omega_1\times\cdots\times\Omega_s\to\mathbb{R}$ satisfies
    $$|f(x_1,\ldots,x_s)-f(x_1',\ldots,x_s')|\le c_i,$$
    whenever $(x_1,\ldots,x_s)$ and $(x_1',\ldots,x_s')$ differ only on the $i$-th coordinate. Here $c_1,\ldots,c_s$ are constants. Then, the random variable $Z=f(X_1,\ldots,X_s)$ satisfies, for every $\lambda\ge0$, that
    $$\Pr[|Z-\mathbb{E}[Z]|\ge\lambda]\le2\exp{\left(\frac{-2\lambda^2}{c_1^2+\cdots+c_s^2}\right)}.$$
\end{lemma}



\begin{proof}[Proof of \cref{claim-concentration}]
    For every $K\in V(\mathcal{H})\subseteq \binom{[n]}{k}$, define $\mathcal{J}(K)$ as the collection of all copies $\mathcal{J}'$ of $\mathcal{J}$ in $\binom{[n]}{k}$ that contains $K$, i.e.
    $$\mathcal{J}(K)=\left\{\mathcal{J}'\subseteq \binom{[n]}{k}: \mathcal{J}'\text{ is a copy of } \mathcal{J} \text{ and } K\in \mathcal{J}'\right\}.$$
    Given $K\in \mathcal{J}'$, there are $\binom{n-k}{m-k}$ choices for $V(\mathcal{J}')/K$. Therefore,
    $$|\mathcal{J}(K)|= \lambda_{\mathcal{J}}\binom{n-k}{m-k},$$
    where $\lambda_{\mathcal{J}}$ denotes the number of non-equivalent embedding of $\mathcal{J}$ into $\binom{[m]}{k}$ with one edge fixed. According to the definition of $\mathcal{H}$, it follows that
    $$\deg_{\mathcal{H}}(K)=|\{\mathcal{J}'\in \mathcal{J}(K): \mathcal{J}' \text{ forms an edge of } \mathcal{H}\}|=\sum_{\mathcal{J}'\in \mathcal{J}(K)}Y_{\mathcal{J}'}, $$
    where $Y_{\mathcal{J}'}$ is the indicator random variable for the event that $\mathcal{J}'$ forms an edge of $\mathcal{H}$. Note that given $K\in V(\mathcal{H})$,
    $$\Pr[Y_{\mathcal{J}'}=1]=(1-\epsilon)^{|\mathcal{J}|-1}\epsilon^{\binom{m}{k}-|\mathcal{J}|}. $$
    For simplicity, let $p=(1-\epsilon)^{|\mathcal{J}|-1}\epsilon^{\binom{m}{k}-|\mathcal{J}|}$. Hence, $$d:=\mathbb{E}[\deg_{\mathcal{H}}(K)]=\sum_{\mathcal{J}'\in \mathcal{J}(K)}\Pr[Y_{\mathcal{J}'}=1]
    =p\cdot \lambda_{\mathcal{J}}\binom{n-k}{m-k}=cn^{m-k},$$
    where $c$ is constant depending on $\mathcal{J}$, $\epsilon$, $m$ and $k$.

    Recall the definition of $X_K$, the indicator random variable for the event that $K\in\binom{[n]}{k}$ is colored blue. For a fixed $K\in V(\mathcal{H})$, $\deg_{\mathcal{H}}(K)$ can be viewed as a function of the variables $\{X_T:T\in \binom{[n]}{k}/\{K\}\}$, i.e.,
    $$\deg_{\mathcal{H}}(K)=f\left(X_T:T\in \binom{[n]}{k}/\{K\}\right).$$
    Next, we will prove that $\deg_{\mathcal{H}}(K)$ is concentrated around its expectation via \cref{lem:BDI}. For every $T\in \binom{[n]}{k}/\{K\}$ with $|T\cap K|=i\in\{0,1,\ldots,k-1\}$, if we change the color of a such $T$, then $c_T=O(n^{m-2k+i})$ since there are at most $\binom{n-(2k-i)}{m-(2k-i)}$ choices for the remaining vertices in an edge of $\mathcal{H}$ containing $T$ and $K$. Therefore, we have that
    $$\sum_{T\in \binom{[n]}{k}/\{K\}}c_T^2=\sum_{i=0}^{k-1}\sum_{T:|T\cap K|=i}c_T^2=\sum_{i=0}^{k-1}O(n^{k-i+2(m-2k+i)})=O(n^{2m-2k-1}),$$ where the second equality follows the fact that the number of $T$ such that $|T\cap K|=i$ is at most $\binom{k}{i}\binom{n-k}{k-i}=O(n^{k-i})$.

    Applying \cref{lem:BDI} with $\deg_{\mathcal{H}}(K)=f\left(X_T:T\in \binom{[n]}{k}/\{K\}\right)$, we have that
    \begin{align*}
        \Pr[\deg_{\mathcal{H}}(K)\not=(1\pm d^{-\epsilon})d]&=\Pr[|\deg_{\mathcal{H}}(K)-d|\ge d^{1-\epsilon}]\\
        &\le2\exp{\left(-\frac{2d^{2(1-\epsilon)}}{\sum_{T\in \binom{[n]}{k}/\{K\}}c_T^2}\right)}\\
        &=2\exp{\left(-\Theta\left(\frac{n^{2(1-\epsilon)(m-k)}}{n^{2(m-k)-1}}\right)\right)}\\
        &=2\exp{\left(-\Theta\left(n^{1-2\epsilon(m-k)}\right)\right)}.
    \end{align*}
    Hence,
    \begin{align*}
        \Pr\left[\forall K\in V(\mathcal{H}), \deg_{\mathcal{H}}(K)=(1\pm d^{-\epsilon})d\right]
        &=1-\Pr\left[\exists K\in V(\mathcal{H}), \deg_{\mathcal{H}}(K)\ne(1\pm d^{-\epsilon})d\right]\\
        &\ge 1-\sum_{K\in V(\mathcal{H})}\Pr[\deg_{\mathcal{H}}(K)\ne(1\pm d^{-\epsilon})d]\\
        &\ge 1-\binom{n}{k}\cdot 2\exp{\left(-\Theta\left(n^{1-2\epsilon(m-k)}\right)\right)}\\
        &=1-2\exp{\left(\Theta\left(k\ln{n}-n^{1-2\epsilon(m-k)}\right)\right)},
    \end{align*}
   which tends to $1$ as $n\to\infty$. In other words, every vertex in $\mathcal{H}$ is contained in $(1\pm d^{-\epsilon})d$ edges with high probability, completing the proof of the claim.
\end{proof}

\end{document}